\documentclass[letterpaper, 10 pt, conference]{ieeeconf} 

\usepackage{amsfonts,amssymb,amsmath}
\usepackage{graphicx,graphics,epsfig,color}
\usepackage{rgsMacros}
\usepackage{subcaption}
\usepackage{hyperref}

\let\labelindent\relax
\usepackage{enumitem}
\usepackage{balance}

\definecolor{olivegreen}{rgb}{0.14,0.29,0}

\usepackage{caption}
\usepackage{subcaption}

\IEEEoverridecommandlockouts                              
\newtheorem{exe}{Example}
\newtheorem{corol}{Corollary}
\newtheorem{ass}{Assumption}
\newtheorem{proper}{Property}
\newtheorem{defin}{Definition}
\newtheorem{prob}{Problem}
\newtheorem{cla}{Claim}
\newtheorem{rem}{Remark}
\newtheorem{lem}{Lemma}
\newtheorem{prop}{Proposition}
\newtheorem{thm}{Theorem}
\newtheorem{fct}{Fact}
\newenvironment{lemma}{\begin{lem}}{\hfill $\square$ \end{lem}}

\newenvironment{example}{\begin{exe}}{\hfill $\square$ \end{exe}}
\newenvironment{remark}{\begin{rem} \rm}{ \end{rem}}

\newenvironment{theorem}{\begin{thm}}{\hfill $\square$ \end{thm}}

\newif\ifitsdraft

\newtheorem{dwellt}{Condition}

\newif\ifitsdraft

\usepackage{mathtools}
\usepackage{comment}
\definecolor{cadmiumgreen}{rgb}{0.0, 0.42, 0.24}
\usepackage{tikz}
\usetikzlibrary{shapes,positioning,intersections,quotes}
\usepackage{pgfplots}
\pgfplotsset{compat=1.16}

\tikzset{
    cross/.pic = {
    \draw[rotate = 45] (-#1,0) -- (#1,0);
    \draw[rotate = 45] (0,-#1) -- (0, #1);
    }
}
\usepackage{amsfonts,amssymb,amsmath}

\usepackage{xcolor}

\usepackage{subcaption}

\newlength{\overwritelength}
\newlength{\minimumoverwritelength}
\newcommand{\overwrite}[3][red]{%
  \settowidth{\overwritelength}{$#2$}%
  \ifdim\overwritelength<\minimumoverwritelength%
    \setlength{\overwritelength}{\minimumoverwritelength}\fi%
  \stackrel
    {%
      \begin{minipage}{\overwritelength}%
        \color{#1}\centering\small #3\\%
        \rule{1pt}{9pt}%
      \end{minipage}}
    {\colorbox{#1!50}{\color{black}$\displaystyle#2$}}}

\usepackage{color}

\usepackage{pifont}
\usetikzlibrary{arrows.meta,positioning,calc,decorations.pathmorphing}

\tikzset{
  snakeline/.style = {->,thick, decorate, decoration={pre length=0.2cm, post length=0.2cm, snake, amplitude=.4mm, segment length=2mm}, cadmiumgreen},
  block/.style = {draw, fill=blue!20, minimum height=3em, minimum width=3em},
  pinstyle/.style={pin edge={to-,thin,black}},
}

\usetikzlibrary{shapes.misc}

\tikzset{cross/.style={cross out, draw=black, minimum size=2*(#1-\pgflinewidth), inner sep=0pt, outer sep=0pt},
cross/.default={1pt}}

\makeatletter
\newcommand \EmptyArgParse [2]
  {%
    \if\relax\detokenize{#2}\relax
      \def\ProcessedArgument{#1}%
    \else
      \def\ProcessedArgument{#2}%
    \fi
  }
\NewDocumentCommand \mathcircled { >{\EmptyArgParse{red}}O{} O{circle} m }
  {%
    \mathpalette{\mathcircled@b{#1}{#2}}{#3}%
  }
\newcommand\mathcircled@b[4]
  {%
    \tikz[baseline=(math.base)]
      \node[draw,color=#1,#2,inner sep=1pt] (math) {$\m@th#3#4$};%
  }
\makeatother

\title{\LARGE \bf On the Intelligent Proportional Controller Applied to Linear Systems}

\author{M. C. Belhadjoudja, M. Maghenem, and E. Witrant
\thanks{The authors are with Universit\'e Grenoble Alpes, CNRS, Grenoble-INP, GIPSA-lab, F-38000, Grenoble, France (e-mail: mohamed.belhadjoudja@gipsa-lab.fr).   
}}

\begin{document}

\maketitle

\begin{abstract}
We analyze in this paper the effect of the well-known \textit{intelligent proportional controller} on the stability of linear control systems. Inspired by the literature on neutral time-delay systems and advanced-type systems, we derive sufficient conditions on the order of the control system, under which, the used controller fails to achieve exponential stability. Furthermore, we obtain conditions, relating the system's and the control parameters, such that the closed-loop system is either unstable or not exponentially stable. After that, we provide cases where the intelligent proportional controller achieves exponential stability. The obtained results are illustrated via numerical simulations, and on an experimental benchmark that consists of an electronic throttle valve.
\end{abstract}

\section{Introduction} \label{intro}
Model-free control (MFC) aims to regulate control systems with unknown dynamical equations. The MFC that we consider here has been introduced in \cite{fliess2009,fliess2013}; see also \cite{observer_MFC,adaptive_MFC} for more recent formulations. Generally speaking, this approach consists of relating the input and the output by an equation, known as \textit{the ultra-local form}, involving the output (and its time derivatives), the input, and an unknown function lumping whatever is unknown in the system \cite{fliess2009}. As a result, the control input is composed of two parts:  a first part designed to compensate for the unknown function, and a second part that consists of a classical linear controller, usually, a PID controller. The resulting controller is known as the \textit{intelligent PID} controller. This class of controllers has been tested both numerically and experimentally on different classes of systems, such as automotive engines \cite{app1},  automated vehicles \cite{app4} and fault accommodation in greenhouses \cite{app5}. This being said other types of MFC techniques are available in the literature; see \cite{ADRC}.

Due to its easy implementation, as opposed to more advanced control strategies, MFC using intelligent PIDs is increasingly applied. However, despite this growing popularity, the rigorous analysis of these controllers is still at its early stage, to the best of our knowledge. 
Indeed, the stability guarantees for the resulting closed-loop system remain, mostly, unexplored. Some results along this direction have been obtained, for example, in \cite{ubiquityPID}, where links between the sampled intelligent PID controller and the sampled classical PID controller in velocity form are established.  
In \cite{revisitMFC}, the robustness of intelligent PIDs is studied via sensitivity analysis. In \cite{stab_mfc_del}, the discretized closed loop using MFC is shown to coincide with the Euler forward approximation of a certain class of systems. The stability of the latter class of systems is then analyzed. However, these conclusions do not necessarily extend to the original closed-loop system under MFC. On the other hand, in some works, MFC and its intelligent linear controllers have been reinforced via different control techniques. For example, in \cite{MPC_MFC},  MFC is combined with model predictive control. In \cite{Sliding_MFC}, a controller combining MFC and sliding mode control is proposed. Despite their proven efficiency, these techniques are more complex to implement, as opposed to the intelligent PID controller.

In this paper, we prove the efficiency and show the limitations of the intelligent proportional controller (iP) when applied to linear control systems. That is, we prove that applying an intelligent proportional controller to a linear control system reduces to applying a PD controller to a neutral delay system. Hence, the intelligent proportional controller inherits some of the limitations of the classical PD controller. 
More precisely, we derive sufficient conditions on the order of the system, under which, the origin fails to be exponentially stable. Furthermore, we derive sufficient conditions, on the system's parameters and the control gains, under which, the origin is either unstable or fails to be exponentially stable. Then, based on existing results on neutral delay systems, we derive sufficient conditions for exponential stability, which illustrates situations where the iP controller guarantees better results than just asymptotic stability. We illustrate our theoretical results via numerical examples and via an experimental benchmark.  In the latter, we solve the the angle-tracking problem for an electronic throttle valve. 

The paper is organized as follows. The problem statement is in Section \ref{problem_statement}. 
Then, some preliminary results are presented in Section \ref{preliminaries}. Our main results are in Section \ref{main_results}. Numerical examples are in Section \ref{numerical_examples}. Finally, the experimental results are in Section \ref{experiments}. 

\textbf{Notation.} 
We denote by $\mathbb{R}^n$ the set of $n$-uples of real numbers, by $\mathbb{R}_{>0}$ the set of positive real numbers and by $\mathbb{R}_{\geq 0}$ the set of nonnegative real numbers. We let $\mathbb{N} :=\lbrace 0,1,2,...\rbrace$, $\mathcal{Z} := \{ 0, \pm 1, \pm 2, \hdots \}$, and $\mathbb{C}$ be the set of complex numbers. Given $\tau \in \mathbb{R}_{>0}$ and a time-varying function $y$, we write $y_{\tau}(t) := y(t-\tau)$. Given $a\in \mathbb{N}$, we denote the $a^{th}$ derivative of $y$ by $y^{(a)}$, the first derivative by $\dot{y}$, and the second derivative by $\ddot{y}$. Given a matrix $A\in \mathbb{R}^{n\times n}$, we denote by $||A||$ its $2$-norm, by $s(A)$ its spectral abscissa, by $\rho (A)$ its spectral radius and by $\mu (A)$ its logarithmic norm with respect to $||.||$. We denote by $I_{n}$ the identity matrix of dimension $n$, and by $0_{nm}$ the zero matrix of dimension 
$n \times m$.

\begin{figure}
\centering
\includegraphics[scale = 0.9]{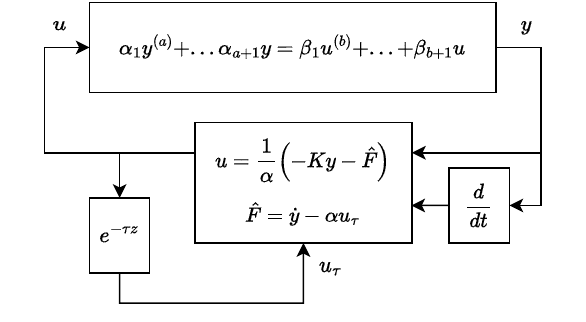}
\caption{The iP control algorithm applied to system \eqref{eq.1}.}
\label{fig1}
\end{figure}

\section{Problem Statement}\label{problem_statement}

We consider linear control systems of the form  
\begin{align} \label{eq.1}
\alpha_{1}y^{(a)}+\hdots +\alpha_{a+1}y=\beta_{1}u^{(b)}+\hdots +\beta_{b+1}u, 
\end{align}
where $y\in \mathbb{R}$ is the measured output, $u\in \mathbb{R}$ is the control input, $a,b\in \mathbb{N}$ with $a>b$, $\alpha_{i},\beta_{j}\in \mathbb{R}$ for $(i,j)\in \lbrace 1,2,\hdots a+1\rbrace \times \lbrace 1,2,\hdots ,b+1\rbrace$, and $\alpha_{1} \beta_1 \neq 0$.

We recall in this section the structure of the intelligent proportional controller applied to system \eqref{eq.1}. For a detailed presentation of MFC approaches with various examples, we refer the reader to \cite{fliess2009} and \cite{fliess2013}. 

Consider the control system in \eqref{eq.1} and let us define, along the trajectories of the system, the time-varying function: 
\begin{align}
    F(t) := \dot{y}(t)-\alpha u(t). 
\end{align}
where $\alpha \neq 0$ is a design parameter. The trajectories of  \eqref{eq.1} can, therefore,  be described by the \textit{ultra-local form} \cite{fliess2009}:
\begin{align}
\dot{y}(t) = \alpha u(t) + F(t). 
\end{align}
The intelligent proportional control law is  given by 
\begin{align} \label{eq.9}
u(t) := \frac{1}{\alpha}\left(-Ky(t)-\hat{F}(t)\right),
\end{align}
where $\hat{F}$ is an estimate of $F$, and $K\in \mathbb{R}$ is a control gain. 

The key idea behind the intelligent proportional control law \eqref{eq.9} is that if $\hat{F} \equiv F$, then the closed-loop system is governed by the equation $\dot{y}=-Ky$, which leads to exponential stability if and only if $K>0$. Various estimates of $F$ are proposed in the literature; 
see for e.g. \cite{fliess2009,fliess2013,
observer_MFC,adaptive_MFC}. 
The one we consider here is given by 
\begin{align}  \label{eq.10}
\hat{F}(t) = \dot{y}(t)-\alpha u_{\tau}(t), \qquad u_{\tau}(t) := u (t-\tau),  
\end{align}
where $\tau \in \mathbb{R}_{>0}$ 
is a time delay. 

\begin{remark}
When $\dot{y}$ is not available for measurements, various methods to approximate it can be found in \cite{derivative1,derivative2,derivative3,fliess2013}.
Although we assume here that $\dot{y}$ is perfectly known, the analysis using its approximations is an interesting perspective to this work.
\end{remark}

The structure of such a controller is illustrated in the block diagram of Figure \ref{fig1}.

Clearly, when applying \eqref{eq.9}-\eqref{eq.10} to \eqref{eq.1}, the resulting closed-loop system is not governed by  $\dot{y}=- K y$. It will most likely involve a time delay (except for specific cases; see Remark \ref{rem1}). The effect of this delay cannot be ignored in general, even if it is sufficiently small; see \cite{smalldelay1,smalldelay2}.
Furthermore, when the closed-loop system has a delay term, we will show that it is either a neutral delay system or an advanced-type system.  As a consequence, we will be able to derive stability and instability guarantees based on some well-known results concerning the latter two classes of systems.

Before presenting our main results, we will recall some preliminaries in the next section.

\section{Preliminaries}\label{preliminaries}

\subsection{Neutral Delay Systems}

Consider the neutral delay system of the form 
\begin{align} \label{eq.2-} 
\alpha_{1}y^{(a)}+\hdots + \alpha_{a+1}y= \beta_{1} y^{(a)}_\tau +\hdots +\beta_{a+1} y_\tau,
\end{align}
for some $a \in \mathbb{N}$ and  $\alpha_{i}, \beta_i \in \mathbb{R}$ for all $i \in \lbrace 1,2,\hdots a+1 \rbrace$ with 
$\alpha_1 \beta_1 \neq 0$. 
Results on well-posedness of system \eqref{eq.2-}, i.e. existence and uniqueness of solutions, can be found in \cite[Chapter 1]{fridman}, \cite[Chapter 9]{haledelay}.

System \eqref{eq.2-} admits the characteristic equation:
\begin{equation}
\label{eq.5} 
\begin{aligned}
0 & = \alpha_1 z^a + \alpha_{2}  z^{a-1} + \hdots + \alpha_{a+1} 
\\ & \qquad - e^{-\tau z} \left(\beta_1 z^a + \beta_2  z^{a-1} + \hdots + \beta_{a+1} \right).  
\end{aligned}
\end{equation}
\begin{lemma}[Neutral root chains {\cite[Item II. Theorem 1]{roots_locali}}] \label{lem1} 
If $\beta_1 \alpha_1 \neq 0$, then equation \eqref{eq.5} has an infinite number of roots given by the sequence:
\begin{equation}
\label{lem1_eq}
\begin{aligned}
z_{k} & = \frac{1}{\tau}\left( \log \left(\bigg| \frac{\beta_1}{\alpha_1} \bigg| \right) + i\left( \text{arg} \left( \frac{\beta_1}{\alpha_1} \right) + 2k\pi \right) \right) 
+ g_k 
\\ & \qquad \qquad \forall k \in \mathcal{Z}, 
\end{aligned}
\end{equation}
where $g_k = o\left(1\right)$; namely, for any $c> 0$, there exists $k_c >0$ such that $|g_k| \leq c$ for all $k \in \mathcal{Z}$ such that $|k| \geq k_c$.

Moreover, besides the sequence $\{z_k \}$, equation \eqref{eq.5} has a finite number of other roots. 
\end{lemma}

\subsection{Advanced-Type Systems}

Systems of \textit{advanced type} are of the form 
\begin{align} \label{eq.2--} 
\alpha_{1}y^{(a)}+\hdots + \alpha_{a+1}y= \beta_{1} y^{(b)}_\tau +\hdots +\beta_{b+1} y_\tau,
\end{align}
for some $a, b \in \mathbb{N}$, with $b > a$, $\alpha_{i} \in \mathbb{R}$ for all $i \in \lbrace 1,2,\hdots a+1 \rbrace$, and $\beta_i \in \mathbb{R}$ for all $i \in \lbrace 1,2,\hdots b+1 \rbrace$, with $\alpha_1 \beta_1 \neq 0$. 

System \eqref{eq.2--} admits a characteristic equation of the form
\begin{equation}
\label{eq.6}  
\begin{aligned}
0 =&~ \alpha_1 z^a + \alpha_{2}  z^{a-1} + \hdots + \alpha_{a+1} 
\\ &~ - e^{-\tau z} \left( \beta_1 z^b +  \beta_2  z^{b-1} + \hdots + \beta_{b+1} \right). 
\end{aligned}
\end{equation}

\begin{lemma}[Unstable root chains {\cite[Item III. Theorem 1]{roots_locali}}]\label{lem2}
If $b>a$ and $\alpha_1 \beta_1 \neq 0$, then the number of roots of \eqref{eq.6} with negative real part is finite. Additionally, \eqref{eq.6} has an infinite number of roots with arbitrarily large non-negative real parts. 
\end{lemma}

\subsection{Neutral Delay Control Systems Subject to PD Control}

We introduce the particular class of neutral delay control systems of the form
\begin{align} \label{eq.2}
  \alpha_{1}\left(y^{(a)}-y_{\tau}^{(a)}\right)+\hdots +\alpha_{a+1}\left(y-y_{\tau}\right) = v, 
\end{align}
  where $v \in \mathbb{R}$ is the control input, $a \in \mathbb{N}$, $\alpha_{i} \in \mathbb{R}$ for $i \in \lbrace 1,2,\hdots a+1\rbrace$, and $\alpha_{1}\neq 0$. 
\ifitsdraft
The characteristic equation of  system \eqref{eq.2}, when $v = 0$, is given by 
the \textit{exponential polynomial} of the form 
\begin{equation}
\label{eq.3}
\begin{aligned}
F(z) & := z^a + \mathcal{K}_1  z^{a-1} + \hdots + \mathcal{K}_a 
\\ & - e^{-\tau z} \left(z^a + \mathcal{K}_1  z^{a-1} + \hdots + \mathcal{K}_a \right),  
\end{aligned}
\end{equation}
where $z \in \mathbb{C}$, $\tau \in \mathbb{R}_{>0}$ is the delay variable,  and $\mathcal{K}_i, i \in \{1,...,a\}$,  are constants depending on the system's parameters.
\fi

System \eqref{eq.2} is said to be (asymptotically, exponentially) stabilizable by PD control of order $(b+1)$ if there exist gains $K_{d1},K_{d2},\hdots,K_{d(b+2)} \in \mathbb{R}$, with $K_{d1}\neq 0$, such that the control law 
\begin{align} \label{eq.4}
v := K_{d1}y^{(b+1)}+K_{d2}y^{(b)}+\hdots+K_{d(b+2)}y 
\end{align}
renders the origin of the resulting closed-loop system (asymptotically, exponentially) stable. We recall here that the closed-loop system is exponentially stable if there exist constants $\kappa, \sigma >0$ such that for any initialization $y([-\tau,0])$, we have 
\begin{align}
|y(t)|\leq \left(\sup_{s\in [-\tau, 0]}|y(s)|\right)\kappa e^{-\sigma t} \quad t\geq 0.
\end{align}

The following two facts are in order. 

\begin{itemize}
\item When $a > b+1$ or $\alpha_1 \neq  K_{d1}$, system \eqref{eq.2} in closed loop with the PD controller \eqref{eq.4} is of neutral type \eqref{eq.2-}. 

\item When system \eqref{eq.2} in closed loop with \eqref{eq.4} is such that 
\begin{align}
a = b+1, \quad
\alpha_1 = K_{d1}, \nonumber
\\
 \exists i\in \lbrace 2,3,...,a+1\rbrace : \alpha_{i}\neq K_{di}, \label{eqad}
\end{align}
the closed-loop system is of advanced type \eqref{eq.2--}. 
\end{itemize}

\section{Main Results} \label{main_results}

We start showing that the closed loop of \eqref{eq.1} using
\eqref{eq.9}-\eqref{eq.10} has the form of \eqref{eq.2} subject to a PD control law in the form of \eqref{eq.4}. As a result, depending on the degree of the control system \eqref{eq.1}, its parameters, and the parameters $(\alpha, K)$ of the intelligent proportional controller, the closed-loop system is either a neutral delay system of the form \eqref{eq.2-} or an advanced-type system of the form \eqref{eq.2--}. In particular, we use Lemmas \ref{lem1} and \ref{lem2} to derive conditions, under which, the closed-loop system either fails to be exponentially stable or is unstable.    

\begin{remark} 
It is well known that if $|\beta_1/\alpha_1| < 1$ in \eqref{eq.2-}, then asymptotic stability is equivalent to exponential stability \cite{kharitonov2005exponential}.
However, this is not true in general. For example, the origin of $\dot{y}-\dot{y}_{\tau} = -y$ is not exponentially stable but it is still asymptotically stable. Therefore, the fact that the intelligent proportional controller fails to guarantee exponential stability in some scenarios does not mean that it cannot achieve asymptotic stability. 
\end{remark}

\subsection{Equivalence between the iP Control of \eqref{eq.1} and the PD Control of \eqref{eq.2}}

We start introducing the following lemma.  

\begin{lemma}\label{thm1}
    A strictly causal system \eqref{eq.1} (with $a > b$) in closed loop with the intelligent proportional controller in \eqref{eq.9}-\eqref{eq.10}  
    can be expressed as a neutral  delay system \eqref{eq.2}, of the same order $a$, subject to a PD controller in \eqref{eq.4} of order $(b+1)$, whose gains $K_{di}$, $i \in \lbrace 1,2,\hdots ,b+2\rbrace$, satisfy
    \begin{align}
        K_{d1} & = -\frac{1}{\alpha}\beta_{1}, \ \ K_{d(b+2)} = -\frac{1}{\alpha}\beta_{b+1}K, ~ \text{and} \nonumber \\ 
        K_{dj} & = -\frac{1}{\alpha}\left(\beta_{j-1} K + \beta_{j} \right) \quad  \forall j \in 
        \{2, \hdots, b+1\}. \label{eq.11} 
  \end{align}
\end{lemma}

\begin{proof}
Using \eqref{eq.9}, we obtain
\begin{align}
&\beta_{1}u^{(b)}+\hdots  + \beta_{b+1}u  = \beta_{1}u_{\tau}^{(b)}+\hdots +\beta_{b+1}u_{\tau} \nonumber\\
&~ -\frac{1}{\alpha}\beta_{1}\left(Ky^{(b)}+y^{(b+1)}\right) -\frac{1}{\alpha}\beta_{2}\left(Ky^{(b-1)}+y^{(b)}\right) \nonumber \\
&~ - \hdots  -\frac{1}{\alpha}\beta_{b+1}\left(Ky+\dot{y}\right), 
\end{align}
which can be rewritten as 
\begin{align}
&\beta_{1}u^{(b)}  + \hdots +\beta_{b+1}u = \nonumber \\ & \alpha_{1}y_{\tau}^{(a)}+\hdots +\alpha_{a+1}y_{\tau}
 -\frac{1}{\alpha}\beta_{1}y^{(b+1)}-\frac{1}{\alpha}\beta_{b+1}Ky \nonumber \\
&~ -\frac{1}{\alpha}\left(\beta_{1}K+\beta_{2}\right)y^{(b)} -\hdots -\frac{1}{\alpha}\left(\beta_{b}K+\beta_{b+1}\right)\dot{y}.
\end{align}    
The closed-loop system is therefore given by 
\begin{align}
& \alpha_{1}\left(y^{(a)}-y_{\tau}^{(a)}\right)+\hdots +\alpha_{a+1}\left(y-y_{\tau}\right) = \nonumber \\
&~-\frac{1}{\alpha}\beta_{1}y^{(b+1)}-\frac{1}{\alpha}\beta_{b+1}Ky \nonumber \\
&~-\frac{1}{\alpha}\left(\beta_{1}K+\beta_{2}\right)y^{(b)} - \hdots -\frac{1}{\alpha}\left(\beta_{b}K+\beta_{b+1}\right)\dot{y},
\end{align}
which is the structure of 
\eqref{eq.2} subject to a PD controller \eqref{eq.4} of order $(b+1)$. 
\end{proof}

\begin{remark} \label{rem1}
From Lemma \ref{thm1}, we conclude that the strictly causal control system in \eqref{eq.1} subject to the iP control in \eqref{eq.9}-\eqref{eq.10} is under one of the following three forms:

\begin{itemize}
\item Advanced-type form \eqref{eq.2--}: if there exist $\{K_{di}\}_{i \in \{1,...,b+2\}}$, such that 
\eqref{eqad}-\eqref{eq.11} hold. 
\item Undelayed form: if there exist $\{ K_{di} \}_{i \in \{1,...,b+2\}}$, with $a=b+1$, \eqref{eq.11} holds,  and $\alpha_{i}=K_{di}$ for all $i \in \{1,2,...,a+1\}$. 
\item Neutral delayed form \eqref{eq.2-}: when none of the conditions in the previous two items hold.
\end{itemize}
\end{remark}

\subsection{Instability and Lack of Exponential Stability}
According to Lemma \ref{thm1}, the iP controller applied to  \eqref{eq.1} inherits the following limitations of the PD controller in \eqref{eq.4} applied to  \eqref{eq.2}. 

\begin{itemize}
\item The PD controller in \eqref{eq.4} cannot guarantee exponential stability if $a > b+1$.

\item When $a = b + 1$, the resulting closed-loop system can be of advanced type, and thus unstable. 

\item Even when the closed-loop system is not of advanced type, when the coefficient of $y^{(a)}_\tau$ is, in norm, larger than one, then the closed-loop system is unstable.  

\item  When the latter coefficient is equal to $1$, the closed-loop system fails to be exponentially stable. 
\end{itemize}

The aforementioned facts are applied to assess the instability of \eqref{eq.1} in closed loop with  \eqref{eq.9}-\eqref{eq.10}.

\begin{theorem}\label{thm2}
Consider system \eqref{eq.1}, with $a > b$, in closed loop with the intelligent proportional controller in \eqref{eq.9}-\eqref{eq.10}. Then, for any delay $\tau > 0$, the following properties hold: 

\begin{itemize}
    \item If $a>(b+1)$, then, for any $\alpha \in \mathbb{R}^*$ and $K \in \mathbb{R}$, the origin of the closed-loop system is not exponentially stable. 

    
    \item If $a=(b+1)$ and $\alpha = -\beta_{1}/\alpha_{1}$, then for any $K\in \mathbb{R}$ such that $\alpha_{a+1} \alpha \neq -  \beta_{b+1}K$ or $\alpha_{i} \alpha \neq -  (\beta_{i-1} K + \beta_i)$ for some $i\in \lbrace 2,3,...,b+1\rbrace $,  the origin of the closed-loop system is unstable. 
   
    \item If $a=(b+1)$ and $\alpha \neq -\beta_{1}/\alpha_{1}$, then, for any $K \in \mathbb{R}$,  the origin of the closed-loop system is unstable if 
    \begin{align} \label{eqthmblast}
    \left|\frac{\alpha_{1}}{\alpha_{1}+\beta_{1}/\alpha}\right|>1, 
\end{align}
and not exponentially stable if
\begin{align} \label{eqthmlast}
    \left|\frac{\alpha_{1}}{\alpha_{1}+\beta_{1}/\alpha }\right|=1.  
\end{align} 
\end{itemize}
\end{theorem}
 
\begin{proof}
 From Proposition \ref{thm1}, we know that the closed-loop of \eqref{eq.1} using \eqref{eq.9}-\eqref{eq.10} is governed by the equation 
 \begin{align}
&\alpha_{1} y^{(a)}+ \alpha_{2} y^{(a-1)} + \hdots + \alpha_{a-b-1}  y^{(b+2)}
\nonumber \\ &
+ \left(\alpha_{a-b}+\frac{1}{\alpha}\beta_{1}\right)y^{(b+1)}
+
\left(\alpha_{a-b+1}
+\frac{(\beta_1 K + \beta_{2})}{\alpha}  \right) y^{(b)}
 \nonumber \\ & + \hdots 
+ \left(\alpha_{a}
+\frac{(\beta_b K + \beta_{b+1})}{\alpha}  \right) \dot{y}
 + \left(\alpha_{a+1}+\frac{\beta_{b+1}K}{\alpha} \right) y\nonumber 
\\ &
-\alpha_{1}y_{\tau}^{(a)}- \hdots -\alpha_{a+1}y_{\tau} = 0. 
\end{align}
The characteristic equation of this system is
\begin{align}
     &z^{a}+ \frac{\alpha_{2}}{\alpha_{1}}z^{a-1} + \hdots +  
\frac{\alpha_{a-b-1}}{\alpha_{1}}z^{b+2} \nonumber 
\\ &
+\frac{\left(\alpha_{a-b}+\beta_{1}/\alpha\right)}{\alpha_{1}}z^{b+1}
\nonumber \\ &
+ \frac{1}{\alpha_1} \left(\alpha_{a-b+1}
+\frac{(\beta_1 K + \beta_{2})}{\alpha}  \right) z^{b}
+ \hdots \nonumber 
\\ & + \frac{1}{\alpha_1} \left(\alpha_{a}
+\frac{(\beta_b K + \beta_{b+1})}{\alpha}  \right) z
+\frac{\left(\alpha_{a+1}+K\left(\beta_{b+1}/\alpha \right) \right)}{\alpha_{1}} \nonumber  
     \\ &
     -e^{-\tau z}\left(z^{a}+ \frac{\alpha_{2}}{\alpha_{1}} z^{a-1} + \hdots +\frac{\alpha_{a+1}}{\alpha_{1}}\right)=0. 
\end{align}
 
If $a>(b+1)$, using 
Lemma \ref{lem1}, we conclude that the system admits infinitely many characteristic roots of the form 
 \begin{equation}
     z_{k} = i\frac{2k\pi}{\tau}+o(1) 
     \  \ \ k=0, \pm 1, \pm 2, \hdots
 \end{equation}
In other words, there exists a sequence of characteristic roots that tends towards the imaginary axis. The origin is therefore not exponentially stable.

 We consider now the case where $a=(b+1)$. Hence, the characteristic equation of the system is given by 
\begin{align}
    & \frac{\left(\alpha_{1}+\beta_{1}/\alpha\right)}{\alpha_{1}}z^{a}
+ \frac{1}{\alpha_1} \left(\alpha_{2}
+\frac{(\beta_1 K + \beta_{2})}{\alpha}  \right) z^{a-1}
+ \hdots \nonumber 
\\ & + \frac{1}{\alpha_1} \left(\alpha_{a}
+\frac{(\beta_b K + \beta_{b+1})}{\alpha}  \right) z
+\frac{\left(\alpha_{a+1}+K\left(\beta_{b+1}/\alpha \right) \right)}{\alpha_{1}} \nonumber 
     \\ &
     -e^{-\tau z}\left(z^{a}+ \frac{\alpha_{2}}{\alpha_{1}} z^{a-1} + \hdots +\frac{\alpha_{a+1}}{\alpha_{1}}\right)=0. \label{eq.16} 
\end{align}
If we have $\alpha = -\beta_{1}/\alpha_{1}$, then the previous equation becomes 
\begin{align}
& \frac{1}{\alpha_1} \left(\alpha_{2}
+\frac{(\beta_1 K + \beta_{2})}{\alpha}  \right) z^{a-1}
+ \hdots \nonumber 
\\ & + \frac{1}{\alpha_1} \left(\alpha_{a}
+\frac{(\beta_b K + \beta_{b+1})}{\alpha}  \right) z
+\frac{1}{\alpha_{1}} 
\left(\alpha_{a+1}+K\frac{\beta_{b+1}}{\alpha} \right) \nonumber 
     \\ &
     -e^{-\tau z}\left(z^{a}+ \frac{\alpha_{2}}{\alpha_{1}} z^{a-1} + \hdots +\frac{\alpha_{a+1}}{\alpha_{1}}\right)=0. \label{eq.17} 
\end{align}
If, in addition, $\alpha_{a+1}\alpha \neq -\beta_{b+1}K$ or there exists $i\in \lbrace 2,3,...,b+1 \rbrace $ such that $\alpha_{i}\alpha \neq -\left(\beta_{i-1}K+\beta_{i}\right)$, then equation \eqref{eq.17} is the characteristic equation of an advanced-type system. As a result, using Lemma \ref{lem2}, we conclude that it possesses infinitely many roots with arbitrarily large real parts, which leads to instability of the closed-loop system.  

If $\alpha \neq -\beta_{1}/\alpha_{1}$, the characteristic equation of the closed-loop system is given by \eqref{eq.16}. 
Hence, if we let 
\begin{align}
\theta := \text{arg} \left(\frac{\alpha_{1}}{\alpha_{1}+\beta_{1}/\alpha}\right)
\end{align}
 and apply Lemma \ref{lem1}, we conclude that the system admits infinitely many characteristic roots of the form 
\begin{align}
    z_{k} = \frac{1}{\tau}\left(\log \left(\left|\frac{\alpha_{1}}{\alpha_{1}+\beta_{1}/\alpha}\right|\right)+i\left(\theta + 2k\pi \right)\right)+o(1),
\end{align}
for $k=0,\pm 1, \pm 2, \hdots$. 
In other words, there exists an infinite number of characteristic roots, for which, the real parts converge to
\begin{align}
    \frac{1}{\tau}\log \left(\left|\frac{\alpha_{1}}{\alpha_{1}+\beta_{1}/\alpha}\right| \right).
\end{align}
The origin of the closed-loop system is therefore unstable if \eqref{eqthmblast} holds, and fails to be exponentially stable if \eqref{eqthmlast} is verified. 
\end{proof}

\begin{remark}
Note that the non-exponential and instability properties in Theorem \ref{thm1} hold for any delay $\tau>0$, although the closed-loop system is exponentially stable when $\tau = 0$. This property is known as the small time-delay effect \cite{smalldelay1,smalldelay2}. This property is due to the transcendental term $e^{-\tau z}$ in the characteristic equation of the system, which generates infinitely many characteristic roots as $\tau$ varies by an infinitesimal amount from zero. 
\ifitsdraft
In this context, we say that the \textit{root continuity argument} is broken. 
\fi
\end{remark}

\subsection{A Stability Result}

Inspired by \cite{chen2000delay}, we propose sufficient conditions to guarantee exponential stability of the closed-loop system using the iP controller.


\begin{theorem} \label{prop2}
Consider system \eqref{eq.1}, with $a=b+1$, in closed loop with the intelligent proportional controller in \eqref{eq.9}-\eqref{eq.10} such that $\alpha \neq -\beta_{1}/\alpha_{1}$, and \begin{align}\label{prop2condition1}
\left|\frac{\alpha_{1}}{\alpha_{1}+\beta_{1}/\alpha} \right|<1.
\end{align}
Then, the origin is exponentially stable if
\begin{align}
&s(\hat{A})<0, \quad \hat{A} := 
    \begin{bmatrix}
        0_{a-1} & I_{a-1} 
        \\ 
        A_{a+1} & \begin{bmatrix} A_{a} &  \hdots & A_{2} \end{bmatrix}
    \end{bmatrix}, \label{prop2condition2}\\
&\tau \left|(\alpha_{a+1}~...~\alpha_{2})\right|+|\alpha_{1}|< |\bar{\alpha}_{1}|, \label{prop2condition3} 
\end{align}
where $s(\hat{A})$ is the spectral abcissa of $\hat{A}$, and 
\begin{align}
0 >&~ \big|(A_{a+1} ~ A_{a} ~... ~A_{2}) \big|
 \left( |\alpha_{1}|
 +\tau |(\alpha_{a+1} ~ \alpha_{a} ~ ...~ \alpha_{2}) | \right) \nonumber \\
&~+|\bar{\alpha}_{1}| \mu (\hat{A}),
\label{prop2condition4}
\end{align}
where $\mu (\hat{A})$ is the logarithmic norm of $\hat{A}$, and for all $i \in \lbrace 1,2,...,b+2 \rbrace$, we have 
 $ A_{i} := (-\bar{\alpha}_{i}+\alpha_{i})/\bar{\alpha}_{1}$, $\bar{\alpha}_{i} := \alpha_{i}-K_{di}$,
and $K_{di}$ is given in \eqref{eq.11}. 
\end{theorem}
\begin{proof}
Since $a=b+1$, then the closed-loop of \eqref{eq.1} using \eqref{eq.9}-\eqref{eq.10} is given by
\begin{equation}\label{prop2proof1}
\bar{\alpha}_{1}y^{(a)}+...+
\bar{\alpha}_{a+1}y = \alpha_{1}y^{(a)}_{\tau}+...+\alpha_{a+1}y_{\tau}. 
    \end{equation}
We define the state vector  
$ Y := \begin{bmatrix}
        y & \dot{y} & \hdots & y^{(a-1)}
    \end{bmatrix}^{\top}. $
Since $\alpha \neq -\beta_{1}/\alpha_{1}$, the closed-loop system \eqref{prop2proof1} admits the following state-space representation  
\begin{equation} \label{prop2proof2}
    \dot{Y} - D\dot{Y}_{\tau} = AY+BY_{\tau}, 
\end{equation}
where 
\begin{align}
    D & := \begin{bmatrix}
        0_{(a-1)} & 0_{a-1}  \\ && \\
        0_{a-1}^\top & \alpha_{1}/\bar{\alpha}_{1}
    \end{bmatrix},~  B := \begin{bmatrix}
        0_{(a-1)a} 
        \\ \\ 
        \begin{bmatrix}
        \frac{\alpha_{a+1}}{\bar{\alpha}_{1}} &  \hdots &\frac{\alpha_{2}}{\bar{\alpha}_{1}} \end{bmatrix} 
    \end{bmatrix}, \nonumber 
  \\   
    A & := 
\begin{bmatrix}
0_{a-1} & I_{a-1} \\ \\
-\frac{\bar{\alpha}_{a+1}}{\bar{\alpha}_{1}}
& 
\begin{bmatrix} 
          -\frac{\bar{\alpha}_{a}}{\bar{\alpha}_{1}} & \hdots & -\frac{\bar{\alpha}_{2}}{\bar{\alpha}_{1}}
\end{bmatrix}
    \end{bmatrix}.
    \end{align}
    
We know from \cite[Theorem (i)]{chen2000delay} that a sufficient condition for exponential stability of \eqref{prop2proof2}, is that the spectral radius of $D$ satisfies $\rho (D)<1$, $\hat{A} = A+B$ is Hurwitz, and 
\begin{align}
&\tau ||B||+||D|| < 1,~ \text{and} \label{prop2proof3} \\
&\mu (\hat{A})+\tau ||\hat{A}^{\top}B|| + ||\hat{A}^{\top}D|| < 0. \label{prop2proof4}
\end{align}

The Schur stability condition $\rho (D)<1$ is equivalent to $|\alpha_{1}/\bar{\alpha}_{1}|<1$, i.e. to condition \eqref{prop2condition1}, and Hurwitz stability of $\hat{A}$ is guaranteed by condition \ref{prop2condition2}. Moreover, observing that
\begin{equation}
    ||B|| = \frac{1}{|\bar{\alpha}_{1}|}
|(\alpha_{a+1}~...~\alpha_{2})|,
\end{equation}
and $||D|| = |\alpha_{1}|/|\bar{\alpha}_{1}|$, we conclude that inequality \eqref{prop2proof3} is satisfied under \eqref{prop2condition3}.  

Finally, since 
\begin{equation}
    \hat{A}^{\top}B = \frac{1}{\bar{\alpha}_{1}} \begin{bmatrix}
        \alpha_{a+1}A_{a+1} &\alpha_{a}A_{a+1} & \hdots & \alpha_{2}A_{a+1}\\ 
        \alpha_{a+1}A_{a} & \alpha_{a}A_{a} & \hdots & \alpha_{2}A_{a}\\
        \vdots & \vdots & \hdots & \vdots \\
        \alpha_{a+1}A_{2} & \alpha_{a}A_{2} & \hdots & \alpha_{2}A_{2}
    \end{bmatrix}
\end{equation}
and 
\begin{equation}
\hat{A}^{\top}D = \frac{\alpha_{1}}{\bar{\alpha}_{1}} \begin{bmatrix}
    0_{a(a-1)} &  \begin{bmatrix}
     A_{a+1}
    \\  \vdots \\
   A_{2}
  \end{bmatrix}
\end{bmatrix},
\end{equation}
it follows that
\begin{equation}
\begin{aligned}
|| \hat{A}^{\top}B|| = &\frac{1}{|\bar{\alpha}_{1}|} 
\big|(\alpha_{a+1}~...~\alpha_{2}) \big| 
 \big|(A_{a+1}~...~A_{2}) \big|, 
\end{aligned}
\end{equation}
and 
\begin{equation}
||\hat{A}^{\top}D|| = \left|\frac{\alpha_{1}}{\bar{\alpha}_{1}}\right|
\big|(A_{a+1}~...~A_{2}) \big|. 
\end{equation}
As a result, inequality \eqref{prop2proof4} is satisfied under condition \eqref{prop2condition4}, which completes the proof. 
\end{proof}

\begin{remark} \label{rem5}
  Note that it is not difficult to adapt the statement of Theorem \ref{thm2} when using the general intelligent PD controller defined recursively by 
   \begin{equation}
        u = u_{\tau} -\frac{1}{\alpha}\left(K_{1}y + K_{2}\dot{y}+ \hdots +K_{\nu} y^{(\nu -1)} + y^{(\nu)} \right),
    \end{equation}
which corresponds to the \textit{ultra-local} form 
\begin{equation}
{y}^{(\nu)}(t) = \alpha u(t) + F(t) \qquad \forall t \geq 0, \quad \nu \in \mathbb{N}. 
\end{equation}
Note that the resulting closed-loop system is of the form 
\begin{align} 
\bar{\alpha}_{1}y^{(\nu + b)}+\hdots +\bar{\alpha}_{\bar{\nu}}y= \alpha_{1} y^{(a)}_\tau +\hdots +\alpha_{a+1} y_\tau,
\end{align}
for $\bar{v}:=b+\nu+1$ and $\bar{\alpha}_{i} \in \mathbb{R}$ for all $i \in \lbrace 1,2,\hdots \nu+b+1 \rbrace$.

Here we distinguish between two cases:
\begin{itemize}
\item When the smallest index $i \in \{1,2,...,b + \nu +1\}$ such that $\bar{\alpha}_i \neq 0$ guarantees
$\nu + b - i + 1 \leq a$, then the tools in Lemmas \ref{lem1} and \ref{lem2} used to study the iP controller can be used. Indeed, the resulting closed-loop system must have one of the forms listed in Remark \ref{rem1}.

\item When $\nu + b - i + 1> a$, the resulting closed-loop system has none of the forms listed in Remark \ref{rem1}. It is rather a delayed differential equation, which, according to \cite{roots_locali}, can only have a finite number of unstable characteristic roots. The stability analysis for this particular case will be considered in future work. 
\end{itemize}
\end{remark}

\section{Numerical Examples} \label{numerical_examples}

In this section, we illustrate our main results via three numerical examples. The simulations are performed on Matlab. Neutral delay systems are simulated using the \textit{ddensd} solver, and advanced-type systems are simulated using an implicit Euler scheme. 

\begin{example}
Consider the first-order equation
\begin{equation} \label{eqexp1}
    \dot{y} = y + u.  
\end{equation}
Note that \eqref{eqexp1} has the form of 
 \eqref{eq.1} with $a=b+1$.  Based on Theorem \ref{thm2}, the origin of \eqref{eqexp1} in closed loop with \eqref{eq.9}-\eqref{eq.10} is unstable if either $\alpha = -1$ and $K \neq -1$, or $\alpha \neq -1$ and 
$$\left|\frac{1}{1+1/\alpha}\right|>1,$$
i.e. $\alpha \in (-\infty, -1/2)/\lbrace -1\rbrace $. Indeed, if we take $\alpha = -1$, then the closed-loop system is given by
\begin{equation}
\dot{y}_{\tau} = -\left(1+K\right)y+y_{\tau}. \nonumber
\end{equation}
This is a differential equation of advanced type if $K\neq -1$. Its characteristic equation possesses infinitely many roots with arbitrarily large real part, leading to instability.

Let us take $\tau = 0.01$ and $K=100$. The output is initialized to $y(t)=e^t$ on $[-\tau,0]$ and the control input is initialized to $u(t)=0$ on $[-\tau,0]$. The simulation result is shown in Figure \ref{fig1}.
\begin{figure}
    \centering
    \includegraphics[scale=0.21]{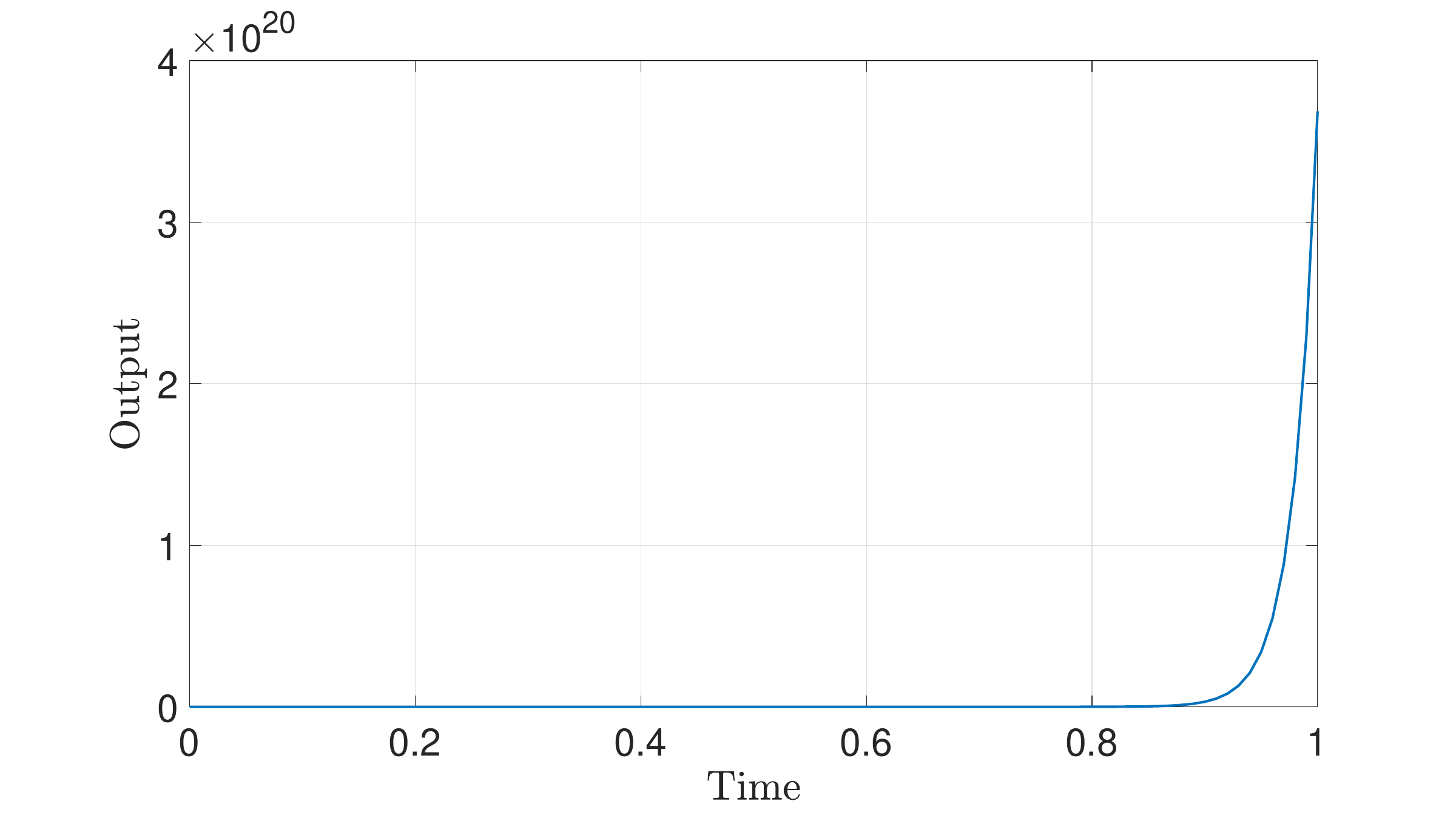}
    \caption{The closed-loop response for $\alpha = -1$ and $K=100$.}
    \label{fig1}
\end{figure}

If we take $\alpha = -2$, the closed-loop system becomes 
\begin{equation}
    \dot{y}-2\dot{y}_{\tau} = \left(2+K\right)y -2y_{\tau}, \nonumber \nonumber 
\end{equation}
which is a neutral delay system whose characteristic equation admits the chain of roots
$z_{k} = \frac{1}{\tau}\left(\log 2 + i2k\pi\right)+o(1)$
for all $k \in \mathcal{Z}$. 
There are, therefore, infinitely many characteristic roots with real parts arbitrarily close to $(\log 2)/\tau$, leading to instability. 
We plot in Figure \ref{fig2} the 
closed-loop response for the same initialization as before, for different values of $\tau$ and for  $K=10$. As expected, by decreasing the time delay, the divergence of the output increases. This is due to the fact that the limit towards which tends the real part of the characteristic roots is inversely proportional to $\tau$. 

\begin{figure}
    \centering
    \includegraphics[scale=0.21]{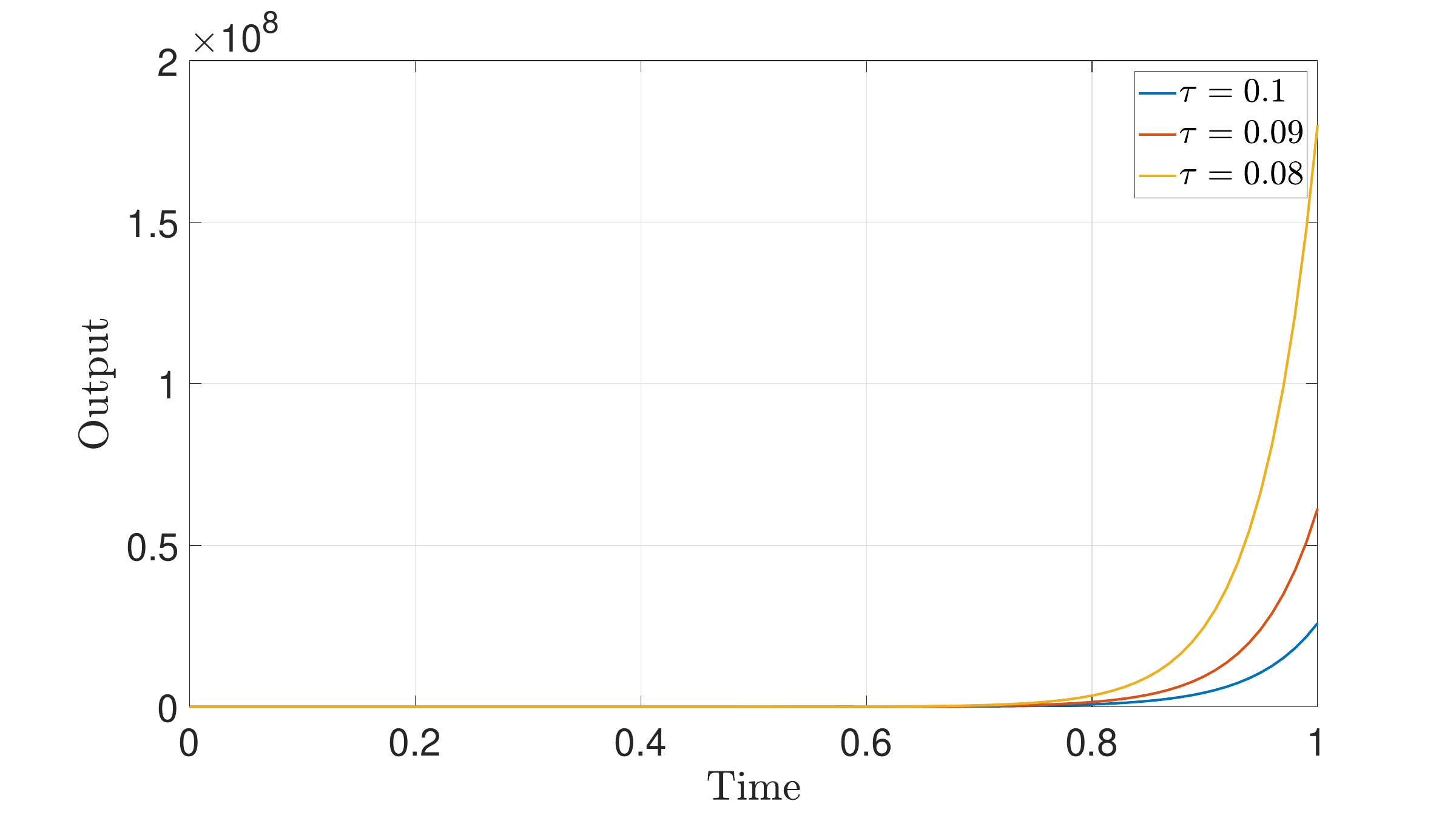}
    \caption{The closed-loop response for $\alpha = -2$ and $K=10$.}
    \label{fig2}
\end{figure}
\end{example}

\begin{example}
We consider now the second-order system 
\begin{equation}\label{egprelast}
    \ddot{y}=y+u, 
\end{equation}
which has the form of \eqref{eq.1} with  $a>(b+1)$. Based on Theorem \ref{thm2}, the origin of system \eqref{egprelast} in closed-loop is not exponentially stable. The closed-loop system is given by 
\begin{equation}
    \ddot{y}-\ddot{y}_{\tau}+\frac{1}{\alpha}\dot{y}+\left(-1+\frac{K}{\alpha}\right)y+y_{\tau}=0, \nonumber 
\end{equation}
which is a neutral delay system whose characteristic equation possesses infinitely many roots arbitrarily close to the 
imaginary axis. 

Although the origin is not exponentially stable, we may still observe convergence towards the origin or boundedness. For example, by setting $\alpha = 0.1$, $K=5$, $\tau = 0.1$, and by considering the initialization $y(t)=e^t$ and $u(t)=0$ on $[-\tau,0]$, we obtain the plot in Figure \ref{fig3}, which illustrates boundedness of 
$y$ and $\dot{y}$. 
\begin{figure}
    \centering
    \includegraphics[scale=0.21]{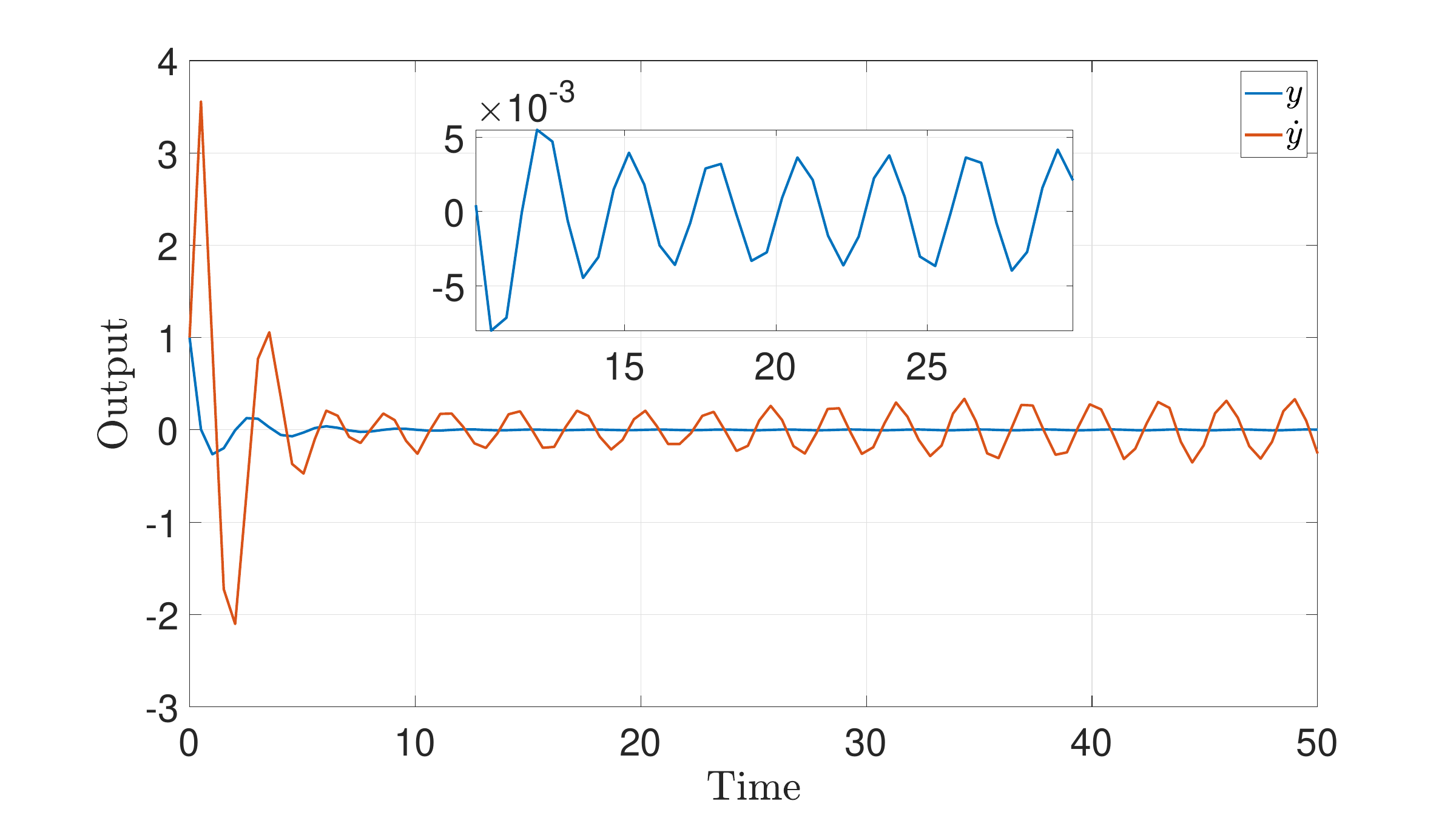}
    \caption{The closed-loop response for $\alpha = 0.1$, $K=5$.}
    \label{fig3}
\end{figure}
\end{example}

\begin{example}  \label{exptun}
Consider the control system
\begin{equation} \label{lastsystem}
\dot{y} + \alpha_2 y = \beta_1 u  \qquad (\alpha_2 , \beta_1) \in \mathbb{R}\times \mathbb{R}^{*}.
\end{equation}
Using the controller in \eqref{eq.9}-\eqref{eq.10}, the resulting closed-loop system is governed by
\begin{equation} \label{preprelast}
\left(1+\frac{\beta_1}{\alpha}\right)\dot{y}-\dot{y}_{\tau}=\left(- \alpha_2 - \frac{\beta_1}{\alpha}K\right)y + \alpha_2 y_{\tau}. 
\end{equation}

First, we would like to avoid having $\alpha + \beta_1 = 0$, which could lead to a closed-loop system of advanced type and thus to instability. For this reason, we set $\alpha := \theta \sign (\beta_1)$, for some $\theta > 0$. Hence, equation \eqref{preprelast} can be expressed as 
\begin{equation}
    \dot{y}-\frac{\theta}{\theta +|\beta_1|}\dot{y}_{\tau}=
   - \frac{\theta \alpha_2 + |\beta_1| K}{\theta+|\beta_1|} y +\frac{ \theta \alpha_2}{\theta+|\beta_1|}y_{\tau}. \nonumber 
\end{equation}

Now, we provide conditions on $\alpha$, $\tau$, and $K$ such that \eqref{prop2condition1}-\eqref{prop2condition4} hold. 
First, condition \eqref{prop2condition1} reduces to $\left|\frac{\theta}{\theta + |\beta_1|}\right|<1. $
According to Theorem \ref{thm1}, the latter inequality is actually a necessary condition for exponential stability, which is verified since  $\theta >0$. 
Second, condition \eqref{prop2condition2} requires that 
$\frac{-|\beta_1| K}{\theta +|\beta_1|} < 0, $
which is verified provided that $K>0$. Third, condition \eqref{prop2condition3} reduces to  
$ \tau |\alpha_2| \theta <|\beta_1|, $
which is satisfied by taking either $\tau$ or $\theta$ sufficiently small. Finally, condition \eqref{prop2condition4} is verified provided that 
\begin{equation}
\label{eqexplast}
\begin{aligned} 
& - \left(1+|\beta_1|/\theta\right)\left(
\alpha_2 + (\beta_1/\theta)K\right) \\ & +\tau |\alpha_2| \left|\alpha_2+(|\beta_1|/\theta)K \right|
 +|\alpha_2|<0.  
\end{aligned}
\end{equation}
For notation simplicity, we let $m:=\alpha_2 + (\beta_1/\theta)K$ and choose either 
$K$ sufficiently large or $\theta$ sufficiently small to have $m>0$. Hence, inequality \eqref{eqexplast} becomes 
\begin{equation}
    \left[ - |\beta_1| - \theta +\tau \theta |\alpha_2|\right]|m|+ \theta |\alpha_2|<0. 
\end{equation}
As a result, it is enough to choose $\theta$ sufficiently small so that \eqref{eqexplast} holds. 

Note that, for this example, the knowledge of $\sign (\bar{\beta})$ allows us to tune $\theta$ and $\tau$ to be sufficiently small while maintaining $K/\theta$ not too large, to obtain exponential stability. For a fixed delay $\tau > 0$, we can achieve exponential stability by tuning only $\alpha$ and $K$.

We take $\alpha_{2}=-1$, $\tau = 0.1$, $\beta_{1}=2$. The output is initialized to $y(t)=e^{t}$ for all $t\in [-\tau,0]$ and the control input to $u(t)=0$ for all $t\in [-\tau,0]$.
We plot in Figure \ref{figlast_example} the closed-loop response for $\alpha = 0.01$ and $K \in \lbrace 1,2,3 \rbrace$, under which, the sufficient conditions for exponential stability are verified. 
Furthermore, we show in the same figure the unstable closed-loop response for $K=10$ and $\alpha=1000$. Note that, when ignoring the delay, the latter gains lead to exponential stability of the closed-loop system $\dot{y} = - Ky$.
\begin{figure}
\centering
\includegraphics[scale=0.21]{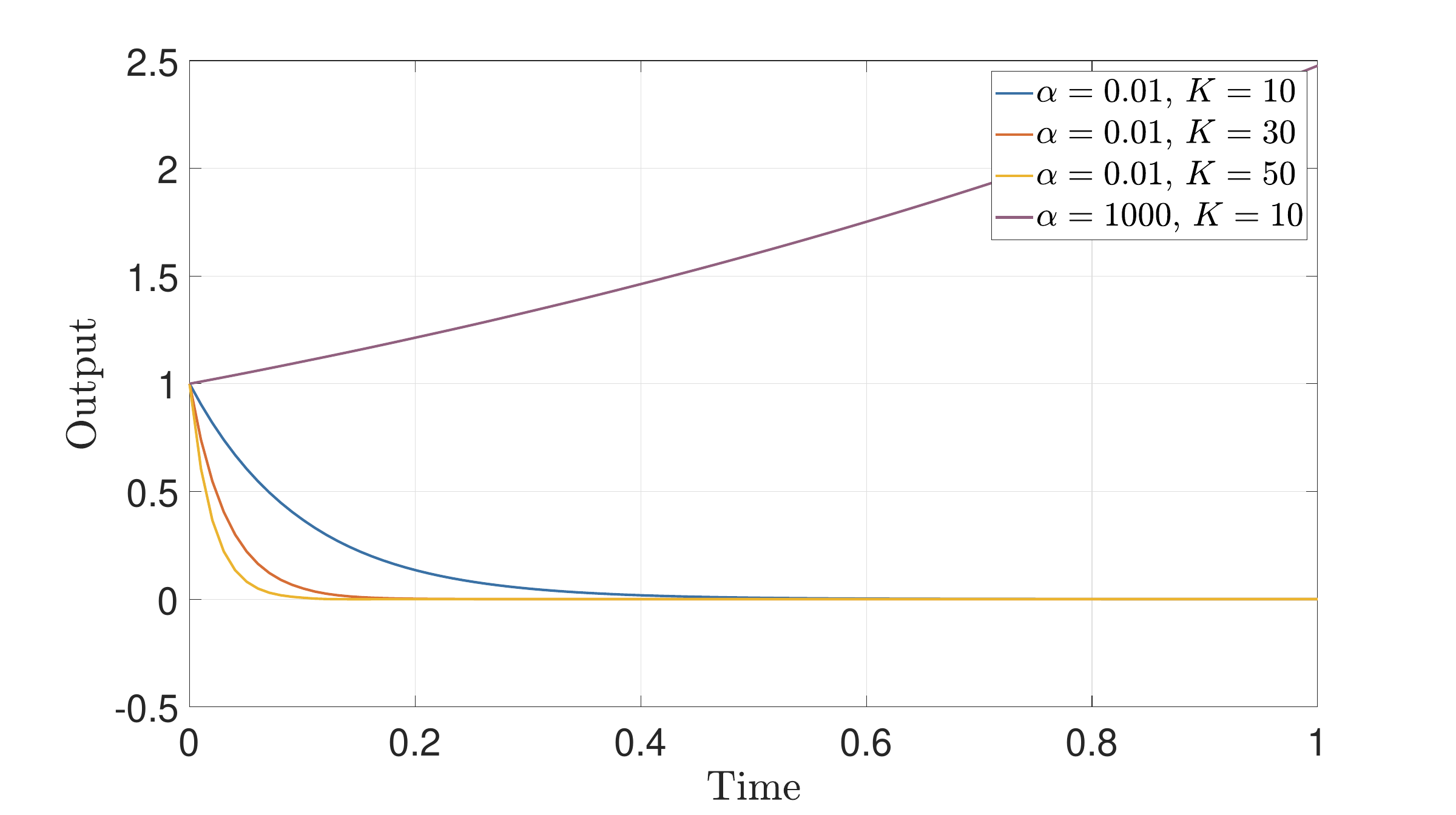}
\caption{The closed-loop response for different $(\alpha, K)$.}
\label{figlast_example}
\end{figure}
\end{example}

\section{Experimental Results} \label{experiments}
In this section, we illustrate some of the obtained theoretical results on an experimental benchmark, which consists of the electronic throttle valve depicted in Figure \ref{valve}. 
This is a butterfly valve used for flow-control applications. The valve's commercial reference is 03L128063. The output of the valve is the opening angle $\theta$, which is regulated by imposing an input voltage $u$ on a $12$V DC motor. The motor is controlled using the SHIELD-MD10 board. The input $u$ is applied by generating a PWM signal from an Arduino Mega $2560$ via the Arduino IDE. The output $\theta$ is measured using the magnetic angle sensor KMA221. We refer to \cite{witrant_valve} for more details on the experimental test bench.
\begin{figure}
\centering
\includegraphics[scale=0.6]{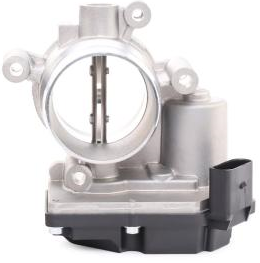}
\caption{Electronic Throttle Valve}
\label{valve}
\end{figure}

Now, given the angular reference $\theta_r$ in Figure \ref{experiment1}, we suppose that, locally, the behavior of the tracking error $y := \theta - \theta_r$ is governed by equation \eqref{eq.1}. 
As a result, we implement the iP controller in \eqref{eq.9}-\eqref{eq.10} with $\tau := 0.05s$, which is also the sampling time of the test bench. 

The output derivative $\dot{y}$ at $t_k$ is approximated using the Euler backward method, i.e., 
$\dot y_k = (y_k - y_{k-1})/\tau$. As a results, the discrete-time controller $u_{k}$ at $t_k$ is given by
\begin{align}
u_{k} := u_{k-1}+\frac{1}{\alpha \tau}\left(-(K\tau+1)y_{k}+y_{k-1}\right).
\end{align}
The initialisation of the controller is $u([-\tau,0])=0$.

In Figure \ref{experiment1}, we plot the valve response for $\alpha := 1/(0.01\tau) = 2000$ and $K := 2\alpha = 4000$, which shows oscillations of $\theta$ around $\theta_r$. Hence, exponential stability is not achieved. If the error $y := \theta - \theta_r$ is governed by \eqref{eq.1} with $a=b+1$, then this lack of exponential stability necessarily means that the sufficient conditions for exponential stability that we derived in Theorem \ref{prop2} are not satisfied.
To verify this fact, we propose to identify such a model, using the  \textit{System-Identification} Matlab Toolbox \cite{ljungtoolbox}. 
For $a = 2$ and $b=1$, we obtained the  model  
\begin{align}
\ddot{y} + 32.16\dot{y} + 1875y = 65.82\dot{u} - 85.89u, \label{model}
\end{align}
whose output, using the input signal generating the response in Figure \ref{experiment1}, matches the response in Figure \ref{experiment1} at the precision of $64.9\%$; see Figure \ref{model_vs_experiment}.
One can check that the sufficient conditions for exponential stability in Theorem \ref{prop2} are not satisfied by  \eqref{model} subject to \eqref{eq.9}-\eqref{eq.10}, mainly because the gain $\alpha$ is too large.
\begin{figure}
    \centering
    \includegraphics[scale=0.21]{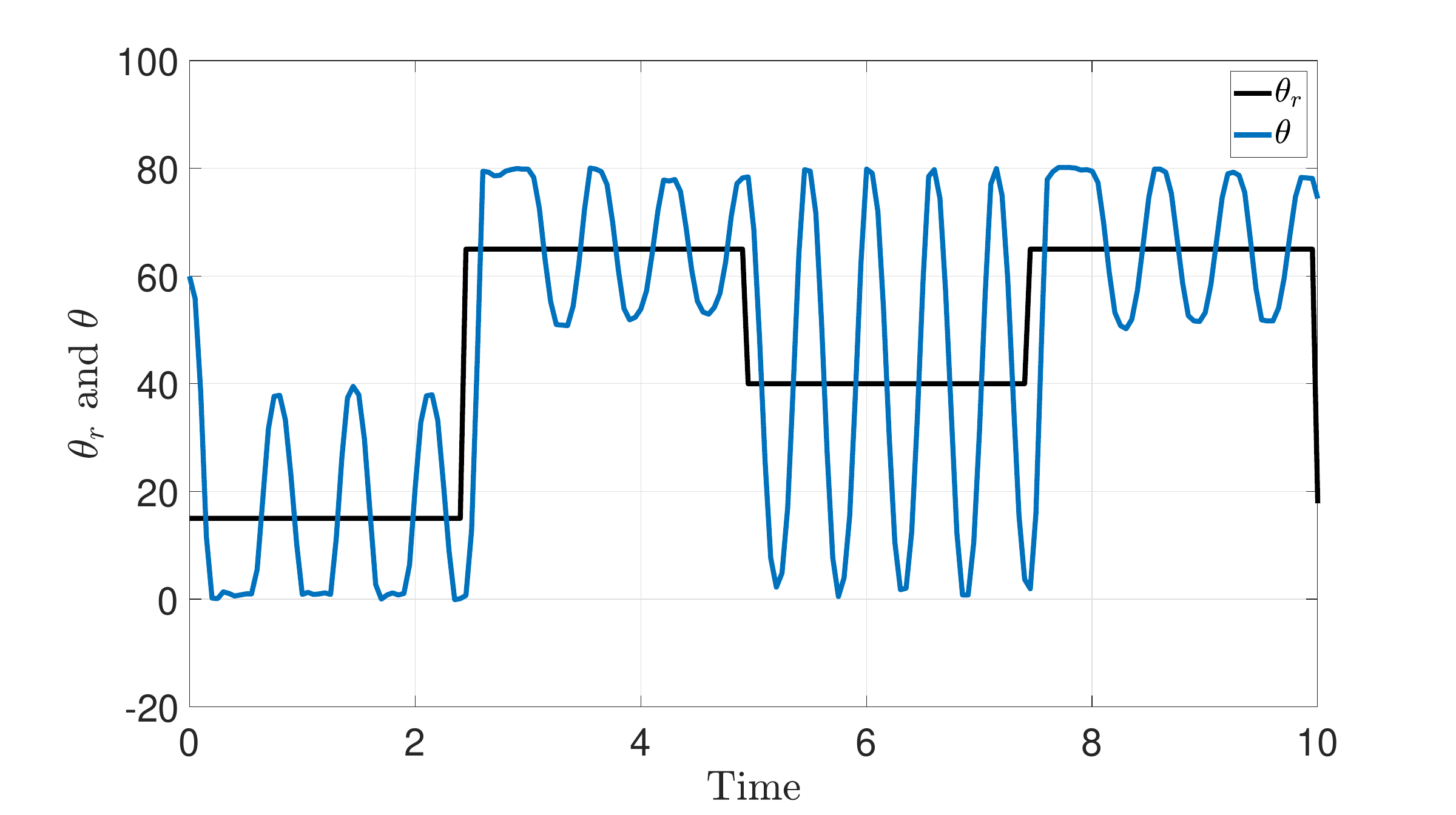}
    \caption{Valve response for $\alpha = 2000$ and $K=4000$. }
    \label{experiment1}
\end{figure}
\begin{figure}
    \centering
    \includegraphics[scale=0.21]{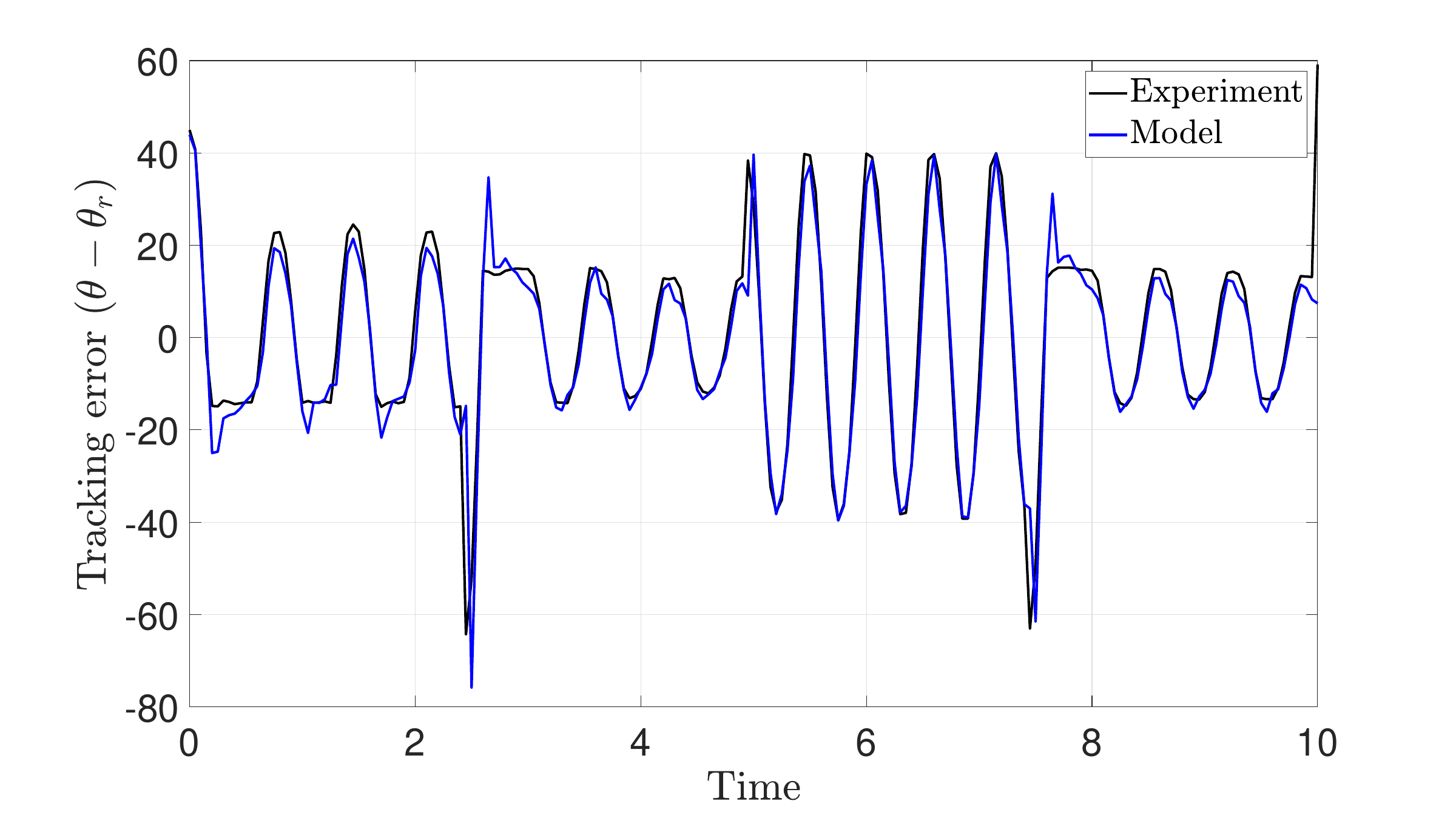}
    \caption{ Valve vs model \eqref{model} response for $\alpha = 2000$ and $K=4000$. }
    \label{model_vs_experiment}
\end{figure}
The intuition we gain from Theorem \ref{prop2} is that it is more likely to achieve exponential stability when $\alpha$ is small than when it is large, as long as the system can be described by \eqref{eq.1} with $a=b+1$. At the same time, we need to make sure that $K/\alpha$ is not too large. Following this intuition, we select $\alpha = 2.5$ and $K=5$. The corresponding valve response is shown in Figure \ref{stable_valve}, where exponential convergence of $\theta$ towards $\theta_r$ is observed.

\begin{figure}
    \centering
    \includegraphics[scale=0.21]{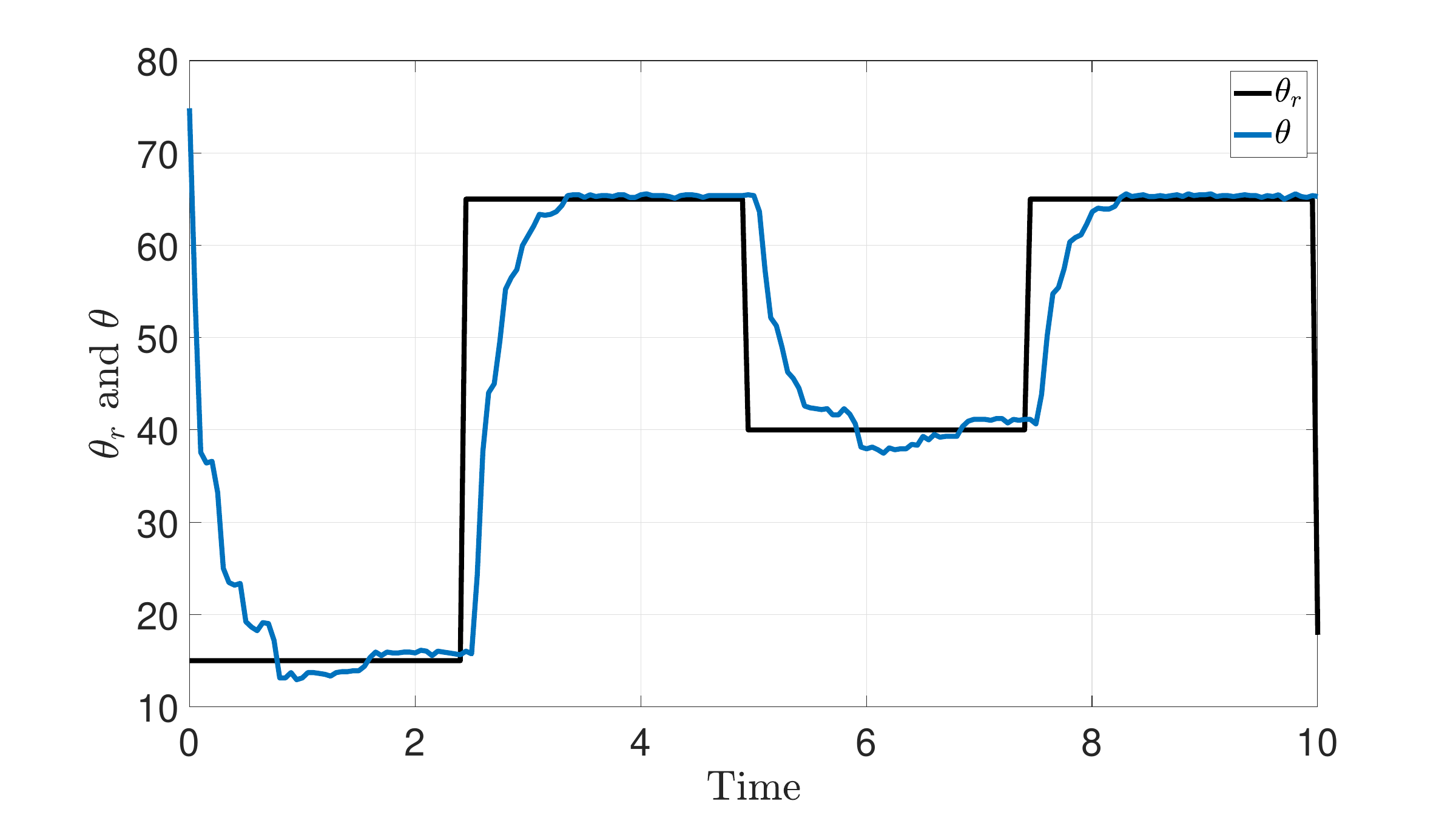}
    \caption{Valve response for $\alpha = 2.5$ and $K=5$.}
    \label{stable_valve}
\end{figure}

\section{Conclusion}
We analyzed the effect of the iP controller on the stability of linear control systems. Inspired by the literature on neutral delay and advanced-type systems, we derived sufficient conditions making the closed-loop system either unstable or not exponentially stable. Some other conditions are derived to guarantee exponential stability.  In particular, we confirm, via theory and experiment, that the iP controller in \eqref{eq.9}-\eqref{eq.10} does not yield to a closed-loop dynamics of the form $\dot{y}=-Ky$. In future work, we would like to study the effect of approximating $\dot{y}$ on the stability of the closed-loop system. 
Furthermore, we would like to consider the more general classes of \textit{intelligent PD} and \textit{intelligent PID} controllers. 

\bibliography{biblio}
\bibliographystyle{ieeetr}

\end{document}